\documentclass[12pt]{amsart}
\usepackage{amsmath}
\usepackage{mathtools}

\usepackage{amsfonts}
\usepackage{amsthm}
\usepackage{autobreak}
\usepackage[english]{babel}
\usepackage{bold-extra}
\usepackage{hyperref}
\usepackage{comment}
\usepackage{filecontents}
\usepackage{mdwlist}
\theoremstyle{definition}
\newtheorem{definition}{Definition}

\newtheorem{remark}[definition]{Remark}
\theoremstyle{plain}
\newtheorem{prop}[definition]{Proposition}

\newtheorem{lemma}[definition]{Lemma}
\newtheorem{theorem}[definition]{Theorem}

\numberwithin{definition}{section}
\allowdisplaybreaks

\def\D{{\mathbb{D}}}

\def\R{{\mathbb{R}}}
\def\C{{\mathbb{C}}}

\def\N{{\mathbb{N}}}

\numberwithin{equation}{section} 
\numberwithin{figure}{section} 
\numberwithin{table}{section} 

\usepackage{multicol}
\usepackage[T1]{fontenc}
\usepackage[utf8]{inputenc}
\usepackage[a4paper]{geometry}
\usepackage{amssymb,amsmath}
\usepackage{version}
\usepackage{mathtools} 
\usepackage{enumerate}
\usepackage{dsfont}
\usepackage{mathdots}
\makeatletter

\usepackage{latexsym}
\usepackage{color}
\definecolor{marron}{rgb}{0.55,0.27,0.07}
\begin{document}

\title[Simultaneous approximation via Blaskche products]{Simultaneous polynomial approximation in Beurling-Sobolev spaces via Blaschke products}


\author{St\'{e}phane Charpentier}
\address{Aix-Marseille Univ, CNRS, I2M, Marseille, France}
\email{stephane.charpentier.1@univ-amu.fr}

\author{Nicolas Espoullier}
\address{Aix-Marseille Univ, CNRS, I2M, Marseille, France}
\email{nicolas.espoullier@etu.univ-amu.fr}

\author{Rachid Zarouf}
\address{Aix Marseille Univ, Lab ADEF, Campus Univ St Jerome, 52 Ave Escadrille Normandie, F-13013 Marseille, France}
\address{CPT, Aix-Marseille Univ., Université de Toulon, Marseille, France}
\email{rachid.zarouf@univ-amu.fr}

\thanks{The first author is supported by the grant ANR-24-CE40-0892-01 of the French National Research Agency
	ANR (project COMOP). The work of Rachid Zarouf is supported by the pilot center Ampiric, funded by the France 2030 Investment Program operated by the Caisse des D\'ep\^ots.}





\begin{abstract}
	Assuming that $\phi(t)=o(t^2)$ as $t\to0$, we establish a lemma on simultaneous polynomial approximation in Orlicz--Beurling--Sobolev spaces $\ell_a^{\phi}$. These spaces, endowed with the Luxemburg norm $\Vert \cdot \Vert _{\ell^{\phi}}$, generalize the classical Beurling--Sobolev spaces $\ell_a^p$ for $p>2$. More precisely, we prove that for every $\varepsilon>0$, every $v\in\mathbb{N}$ and every function $\varphi$ continuous on $\partial\mathbb{D}$, there exist a polynomial $P(z)=\sum_{k=v}^d a_k z^k$ and a compact set $K\subset\partial\mathbb{D}$ with $m(K)>1-\varepsilon$ such that
	\[
	\|P\|_{\ell^{\phi}}\le\varepsilon \quad  \text{and}\quad  \|P-\varphi\|_K\le\varepsilon.
	\]
	The proof relies on a result of independent interest describing the asymptotic behaviour of the Luxemburg norm $\|B^k\|_{\ell^{\phi}}$ of powers of a finite Blaschke product $B$ which is not a monomial. This behaviour is governed by the comparison between $\phi(t)$ and $t^2$ near $0$:  the norms remain bounded when $\phi\asymp t^2$, tend to $0$ when $\phi=o(t^2)$, and diverge to $+\infty$ when $t^2=o(\phi(t))$. A key ingredient in the proof is the qualitative limit $\sup_{j\ge0}|\widehat{B^k}(j)|\to0$ as $k\to\infty$.
	
	As an application of the simultaneous approximation lemma, we derive the existence of functions in $\ell_a^{\phi}$ with universal properties, including Menshov universality of Taylor partial sums and universality with respect to radial boundary limits.
\end{abstract}

\maketitle

\section{Introduction}

Let $\phi: \R_+\to \R_+$ be an Orlicz function, that is an increasing convex function with $\phi(0)=0$. The Orlicz space $\ell_a^{\phi}$ consists of all functions $f$ analytic in the unit disc $\D=\{z\in \C:\, |z|<1\}$ such that
\[
\sum_k \phi\left(\frac{|\hat{f}(k)|}{\lambda}\right)<+\infty,
\]
for some $\lambda >0$, where $\hat{f}(k)$ denotes the $k$-th Fourier coefficient of $f$. Endowed with the Luxemburg norm,
\[
\Vert f\Vert _{\ell^{\phi}}:=\inf\left\{\lambda>0:\,\sum_k \phi\left(\frac{|\hat{f}(k)|}{\lambda}\right)\leq 1\right\},
\]
the space $\ell_a^{\phi}$ is a Banach space. The set of all polynomials is dense in $\ell_a^{\phi}$ if and only if $\phi$ satisfies the $\Delta_2$ condition, namely if there exist constants $K>0$ and $t_0\in (0,1)$ such that $\phi(2t)\leq K\phi(t)$ for all $t\in (0,t_0)$. We refer to \cite{LindenstraussTzafririBookI} for basic properties of Orlicz sequence spaces. Typical examples are given by the Beurling-Sobolev spaces $\ell_a^p$, $p\geq 1$, corresponding to the choice $\phi(t)=t^p$. In this case, $\ell_a^p$ consists of all analytic functions $f$ on $\D$ such that
\[
\Vert f \Vert _{\ell^p}^p:=\sum_k |\hat{f}(k)|^p<+\infty.
\]
The space of all analytic functions $f$ on $\D$ with bounded Fourier coefficients, equipped with the norm
\[
\Vert f \Vert _{\ell^{\infty}}:=\sup_k |\hat{f}(k)|<+\infty,
\]
is denoted by $\ell_a^{\infty}$.

The purpose of this note is to establish a result of simultaneous polynomial approximation in $\ell_a^{\phi}$, under the additional assumption
\[
\phi(t)=o(t^2)\quad \text{as }t\to 0,
\]
using Blaschke products. Given a compact set $K\subset \partial \D$ and a continuous function $h$ on $K$, we write
\[
\Vert h \Vert _K:=\sup_{z\in K}|h(z)|.
\]
\begin{lemma}[Simultaneous approximation in $\ell_a^{\phi}$]\label{lemma1}Let $\phi$ be an Orlicz function such that $\phi(t)=o(t^2)$ as $t\to 0$. For every $\varepsilon>0$, every $v\in \N$ and every continuous function $\varphi$ on $\partial \D$, there exist a polynomial $P=\sum_{k=v}^d a_k z^k$ and a compact set $K\subset \partial \D$, with $m(K)>1-\varepsilon$, such that
	\[
	\Vert P \Vert_{\ell^{\phi}}\leq \varepsilon \quad \text{and}\quad \Vert P-\varphi\Vert_K \leq \varepsilon.
	\]
\end{lemma}
Up to minor differences, this result was already formulated by Kahane and Nestoridis in \cite{KahaneNestoridis2000}. Their proof is essentially based on another approximation result, obtained in \cite{KahaneKatznelson1971} with a probabilistic approach. The interest of our proof lies in illustrating the general principle that such approximation result can be obtained in various settings by means of suitable inner functions. This viewpoint was recently developed in \cite{Limani2024} and further emphasized in \cite{CharpentierEspoullierZarouf2025}.

Simultaneous polynomial approximation has been investigated in a wide range of Fréchet or Banach spaces of analytic functions, including the Fréchet space $H(\D)$ of all analytic functions in $\D$ (where the result reduces to a simple variant of Mergelyan's theorem), growth spaces \cite{Khrushchev1978}, Bergman, Hardy and Dirichlet spaces \cite{BeiseMuller2016,Maronikolakis2022}, weighted $\ell_a^1$ spaces \cite{Charpentier2020},  and the Bloch space \cite{Limani2024}. Versions in several complex variables can be found in \cite{Bayart2005,CharpentierEspoullierZarouf2025}.

More precisely, Lemma \ref{lemma1} will be derived as a consequence of the following result, which is of independent interest. Recall that, given a finite set $\sigma=\{\lambda_1,\ldots,\lambda_m\} \subset \D$, the associated finite Blaschke product $B$ of degree $m$ is defined by
\[
B(z)=\prod_{j=1}^m b_{\lambda_j},
\]
where $b_{\lambda_j}:=\frac{\lambda_j}{|\lambda_j|}\frac{\lambda_j-z}{1-\bar{\lambda_j}z}$ denotes the Blaschke factor associated with $\lambda_j$. If $f$ and $g$ are two positive functions, we write $f\asymp g$  as $t\to 0$ if there exist constants $c,C>0$ and $t_0\in (0,1)$ such that $cf(t)\leq g(t) \leq Cf(t)$ for all $t\in (0,t_0)$.

\begin{theorem}\label{prop:lp_norms}Let $B$ be a finite Blaschke product that is not a power of $z$.
	\begin{enumerate}\item If $\phi \asymp t^2$  as $t\to 0$, then $\Vert B^k \Vert _{\ell^\phi} \asymp 1$, with constants independent of $k \in \N$ (and equality when $\phi(t)=t^2$);
		\item If $\phi(t)=o(t^2)$ as $t\to 0$, then $\Vert B^k \Vert _{\ell^\phi} \to 0$ as $k\to \infty$;
		\item If $t^2 =o(\phi(t))$ as $t\to 0$, then $\Vert B^k \Vert _{\ell^\phi} \to \infty$ as $k\to \infty$.
	\end{enumerate}
\end{theorem}

Estimates of $\ell^p$-norms, $p\in [1,+\infty]$, of powers of Blaschke factors have been studied for a rather long time due to their numerous applications, for example to group algebras \cite{Kahane1956}, crystalline measures \cite{Meyer2021},  composition operators \cite{BlyudzeShimorin1996,LLQR2025,LLQRpreprint2025} and function theory \cite{BorichevFouchetZarouf2025I}. Sharp asymptotics of $\Vert b_{\lambda}^k\Vert_{\ell^p}$, $p\in [1,+\infty]$, were obtained in \cite{SzehrZarouf2020}. In \cite{BorichevFouchetZarouf2025II}, the asymptotic behaviour of $\Vert B^k \Vert _{\ell^{\infty}}$ for finite Blaschke products $B$ is investigated and applied to problems in matrix analysis related to Sch\"affer's question on norms of inverses. From this perspective, simultaneous polynomial approximation in $\ell_a^{\phi}$ can be viewed as a new application of this theory.

\medskip

As shown in \cite{KahaneNestoridis2000}, an application of a result akin to Lemma \ref{lemma1} implies that quasi-all functions in $\ell_a^{p}$ are Menshov universal. More precisely:

\begin{theorem}[\cite{KahaneNestoridis2000}]Let $p>2$. There exists a dense $G_{\delta}$-subset $\mathcal{U}_p$ of $\ell_a^p$ such that every function $f=\sum_ia_iz^i$ in $\mathcal{U}_p$ is Menshov universal, meaning that for any measurable function $\varphi$ on $\partial \D$, there exists an increasing sequence of integers $(d_n)_n$ such that
	\[
	\sum_{i=0}^{d_n}a_iz^i \to \varphi\quad \text{as }n\to \infty,\, \text{a.e. on } \partial \D.
	\]
\end{theorem}

A general approach, pertaining to the theory of universality for sequences of operators, shows that establishing results of \emph{generic} universality of various types (e.g., for Taylor partial sums as above, or with respect to boundary radial limits as in \cite{CharpentierEspoullierZarouf2025}) in a Banach or a Fréchet space $X$ of functions analytic in $\D$, in which the set of polynomials is dense, reduces to proving a simultaneous approximation lemma for $X$, such as Lemma \ref{lemma1}. This machinery is by now classical and is developed in detail in \cite{BayartGrosseErdmanNestoridisPapadimitropoulos2008}; see also Section~3 in \cite{CharpentierEspoullierZarouf2025} for a recent and concise exposition. It is also classical that, even if we drop the assumption that the set of polynomials is dense in $X$, simultaneous approximation arguments still allow one to construct universal functions in $X$; however, one may then lose information on the largeness of the set of such functions in the sense of Baire category theorem. Therefore, we may state, as a consequence of Lemma \ref{lemma1} and without further elaboration, the following universality result in $\ell_a^{\phi}$.

\begin{theorem}\label{thm-applic-2}Let $\phi$ be an Orlicz function such that $\phi(t)=o(t^2)$ as $t\to 0$. There exists a function $f=\sum_i a_iz^i$ in $\ell_a^{\phi}$ that satisfies the following two properties:
	\begin{enumerate}\item Let $(r_n)_n$ be a sequence of real numbers in $(0,1)$ that converges to $1$. Given any measurable function $\varphi$ on $\partial \D$, there exists an increasing sequence $(n_l)_l$ such that, for any $w\in \D$ and a.e. $\zeta \in \partial \D$,
		\[
		f(r_{n_l}(\zeta -w)+w) \to \varphi (\zeta)\quad \text{as }l\to \infty;
		\]
		\item Given any measurable function $\varphi$ on $\partial \D$, there exists an increasing sequence $(d_n)_n$ such that, for a.e. $\zeta \in \partial \D$,
		\[
		\sum_{i=0}^{d_n}a_i\zeta^i \to \varphi(\zeta)\quad \text{as }n\to \infty.
		\]
	\end{enumerate}
	If we assume that $\phi$ satisfies the $\Delta_2$ condition, then the set of such functions is dense and $G_{\delta}$ in $\ell_a^{\phi}$.
\end{theorem}

Note that this theorem exhibits a certain form of optimality, since it is well known that any function in $\ell_a^2$ admits radial boundary limits and has an almost everywhere convergent Fourier series on $\partial \D$. For a sample of results of this type in various settings, we refer to \cite{Bayart2005,BayartGrosseErdmanNestoridisPapadimitropoulos2008,BeiseMuller2016,Charpentier2020,CharpentierEspoullierZarouf2025,Khrushchev2020,Limani2024,Maronikolakis2022}.

The paper is organised as follows. In Section~2, we prove Theorem \ref{prop:lp_norms}. Section~3 is devoted to the proof of Lemma \ref{lemma1}.

\section{Asymptotics of $\ell_a^{\phi}$-norms of powers of finite Blaschke products}\label{section-1}
We shall prove Theorem~\ref{prop:lp_norms} by exploiting asymptotics of the $\ell^{\infty}$-norm of powers of a finite Blaschke product $B$ which is not a monomial. Sharp quantitative asymptotics for $\|B^k\|_{\ell^{\infty}}$ as $k\to \infty$ were obtained by Borichev--Fouchet--Zarouf in \cite[Theorem~1]{BorichevFouchetZarouf2025II} (see Remark \ref{remark-BFZ} below). In the present note, and more specifically for the proof of Assertions~2 and~3 of  Theorem \ref{prop:lp_norms}, we only require the qualitative limit
\[
\|B^k\|_{\ell^{\infty}}\to 0\quad\text{as }k\to \infty.
\]
For the reader's convenience, we therefore provide a direct proof of this fact, relying solely on a standard van der Corput estimate, which already plays a central role in the proof of Theorem~1 in [6].

\begin{lemma}[van der Corput, e.g., Lemma~2.2 in \cite{IvicZetaFunction2003}; special case $G\equiv1$ of Lemma~5 of \cite{BorichevFouchetZarouf2025II}]\label{lem:vdc}Let $F$ be a real $C^2$ function on an interval $[a,b]$ such that either $F''(x)\ge m>0$ for all $x\in[a,b]$ or $F''(x)\le -m<0$ for all $x\in[a,b]$. Then
	\[
	\left|\int_a^b e^{iF(x)}\,dx\right|\le \frac{8}{\sqrt{m}}.
	\]
\end{lemma}

In order to apply Lemma \ref{lem:vdc}, we introduce the following notation. Let $\psi_{B}(\theta)$ denotes the continuous argument (determined
modulo $2\pi$) of $B(e^{{\rm i}\theta})$, so that
\[
B(e^{{\rm i}\theta})=\exp\left({\rm i}\psi_{B}(\theta)\right).
\]
Let $(\xi_{\ell})_{\ell=1}^{s}$ be the sequence of consecutive
zeros of $\psi_{B}''$ on $[0,2\pi)$, with respective multiplicities
$(N_{\ell}-2)_{\ell=1}^{s}$, $N_{\ell}\ge3$. Thus $0\le\xi_{1}<\xi_{2}<\ldots<\xi_{s}<2\pi$ and, for each $\ell=1,\ldots,s$,
\[
\psi''_{B}(\xi_{\ell})=\ldots=\psi_{B}^{(N_{\ell}-1)}(\xi_{\ell})=0,\qquad\psi_{B}^{(N_{\ell})}(\xi_{\ell})\not=0.
\]
Note that the finitess of the set $(\xi_l)_{l=1}^s$ of zeros of $\psi_B''$ is equivalent to the fact that $B$ is not a constant multiple of a power of $z$.

\begin{prop}\label{prop:linfty_decay}
	Let $B$ be a finite Blaschke product that is not a power of $z$. Then $\|B^k\|_{\ell^{\infty}}\to 0$ as $k\to\infty$.
\end{prop}

\begin{proof}
	For $j\ge0$, using $B(e^{i\theta})=e^{i\psi_B(\theta)}$, we write
	\[
	\widehat{B^k}(j)=\frac1{2\pi}\int_0^{2\pi}\exp\bigl(i(k\psi_B(\theta)-j\theta)\bigr)\,d\theta.
	\]
	Let $\xi_1<\cdots<\xi_s$ be the (finite) sequence of zeros of $\psi_B''$ on $[0,2\pi)$, as introduced above, and set $\xi_0=0$, $\xi_{s+1}=2\pi$. Fix $\varepsilon>0$ and split the integral as a sum over the intervals $[\xi_\ell,\xi_{\ell+1}]$, $\ell=0,\dots,s$. On each such interval, decompose
	\begin{multline*}
	\int_{\xi_\ell}^{\xi_{\ell+1}} e^{i(k\psi_B(\theta)-j\theta)}\,d\theta
	=\int_{\xi_\ell}^{\xi_\ell+\varepsilon}\!\! e^{i(k\psi_B(\theta)-j\theta)}\,d\theta \\+\int_{\xi_\ell+\varepsilon}^{\xi_{\ell+1}-\varepsilon}\!\! e^{i(k\psi_B(\theta)-j\theta)}\,d\theta+\int_{\xi_{\ell+1}-\varepsilon}^{\xi_{\ell+1}}\!\! e^{i(k\psi_B(\theta)-j\theta)}\,d\theta.
	\end{multline*}
	The first and third terms have absolute value at most $\varepsilon$. For the middle term, set $F(\theta)=k\psi_B(\theta)-j\theta$, so that $F''(\theta)=k\psi_B''(\theta)$. Since $\psi_B''$ has no zeros on $(\xi_\ell,\xi_{\ell+1})$, its sign is constant there. Moreover
	\[
	m_{\ell,\varepsilon}:=\min_{\theta\in[\xi_\ell+\varepsilon,\xi_{\ell+1}-\varepsilon]}|\psi_B''(\theta)|>0.
	\]
	Applying Lemma~\ref{lem:vdc} on $[\xi_\ell+\varepsilon,\xi_{\ell+1}-\varepsilon]$ yields
	\[
	\left|\int_{\xi_\ell+\varepsilon}^{\xi_{\ell+1}-\varepsilon} e^{iF(\theta)}\,d\theta\right|\le \frac{8}{\sqrt{k\,m_{\ell,\varepsilon}}}.
	\]
	Summing over $\ell=0,\dots,s$ and dividing by $2\pi$, we obtain
	\[
	\big|\widehat{B^k}(j)\big|\le \frac{1}{2\pi}\left(2(s+1)\varepsilon+\frac{8(s+1)}{\sqrt{k\,M_\varepsilon}}\right),\qquad
	M_\varepsilon:=\min_{0\le\ell\le s}m_{\ell,\varepsilon}>0.
	\]
	The right-hand side is independent of $j$, so taking the supremum over $j\ge0$ and then letting $k\to\infty$ shows $\limsup_{k\to\infty}\|B^k\|_{\ell^{\infty}}\le \frac{s+1}{\pi}\,\varepsilon$. Since $\varepsilon>0$ is arbitrary, this proves $\|B^k\|_{\ell^{\infty}}\to0$.
\end{proof}

\begin{remark}\label{remark-BFZ}Proposition~\ref{prop:linfty_decay} provides the qualitative input required in the present note. For context, and in order to highlight the sharp result \cite[Theorem~1]{BorichevFouchetZarouf2025II}, we recall below an asymptotically sharp estimate for $M_k$. This estimate constitutes a quantitative improvement of Proposition~\ref{prop:linfty_decay}; it will not be used in the sequel. Its proof is a technically more involved refinement of the previous one.
	
Given two sequences $(a_k)_k$ and $(b_k)_k$ of positive real numbers, we write $a_k\asymp b_k$ as $k\to \infty$ if there exist two constants $c,C>0$ such that $ca_k \leq b_k \leq Ca_k$ for all sufficiently large $k$.
\end{remark}

\begin{theorem}[Theorem~1 in \cite{BorichevFouchetZarouf2025II}]\label{thm:upper_bd} Let $B$ be
	a finite Blaschke product which is not a power of $z$, and
	let $\psi_{B}$, $(\xi_{\ell})_{\ell=1}^{s}$, $(N_{\ell})_{\ell=1}^{s}$
	be defined as above. Then we have 
	\[
	\|B^{k}\|_{\ell^{\infty}}\asymp k^{-1/N}\quad\text{as }k\to \infty,
	\]
	where $N=N_{B}=\max_{1\le\ell\le s}N_{\ell}$.
\end{theorem}

\medskip

Let us now derive Theorem \ref{prop:lp_norms} from Proposition~\ref{prop:linfty_decay}.

\begin{proof}[Proof of Theorem \ref{prop:lp_norms}]
	Let $L^{2}(\partial\mathbb{D})$ be the usual Hilbert space of all measurable Lebesgue square integrable functions on $\partial\mathbb{D}$, equipped with the standard inner product
	\[
	\left\langle f,g\right\rangle :=\int_{-\pi}^{\pi}f(e^{i\theta})\overline{g(e^{i\theta})}\frac{{\rm d}\theta}{2\pi}=\sum_{j\geq0}\hat{f}(j)\overline{\hat{g}(j)}.
	\]
	For $p=2$ and any $k\geq 0$, it is clear that 
	\[
	\Vert B^{k}\Vert_{\ell^{2}}^{2}=\left\langle B^{k},B^{k}\right\rangle =1,
	\]
	which gives the first assertion when $\phi(t)=t^2$. Assume now that $\phi\asymp t^2$ as $t\to0$. It is clear that $\ell_a^{\phi}$ and $\ell_a^2$ coincide. Further, we shall see that there exist two constants $k_{\phi},K_{\phi}>0$ such that
	\begin{equation}\label{eqnormsequiv}
		k_{\phi}\Vert f \Vert _{\ell^{2}} \le \max\left(\Vert f \Vert_{\ell^{\phi}},\Vert f \Vert _{\ell^{\infty}}\right)\le K_{\phi}\Vert f \Vert _{\ell^{2}},\quad f\in \ell_a^{2}(=\ell_a^{\phi}).
	\end{equation}
	Indeed, the assumption on $\phi$ allows us to choose $t_0\in (0,1)$ and constants $k,K>0$ such that $kt^2\le\phi(t)\le Kt^2$ for $t\in[0,t_0]$. Fix $f\in \ell_a^2$ and $\lambda  \geq \frac{\max\left(\Vert f \Vert_{\ell^{\phi}},\Vert f \Vert _{\ell^{\infty}}\right)}{t_0}$. Then
	\[
	\frac{k}{\lambda^2}\Vert f \Vert _{\ell^2}^2 \le \sum_{j\ge0}\phi\left(\frac{|\widehat{f}(j)|}{\lambda}\right)\le 1.
	\]
	Thus $\lambda \geq \sqrt{k}\Vert f \Vert _{\ell^2}$ hence the first inequality of \eqref{eqnormsequiv}, setting $k_{\phi}=\sqrt{k}t_0$. For the other inequality, fix $f\in \ell_a^2$ and $\lambda  \geq \frac{\sqrt{K}\Vert f \Vert_{\ell^2}}{t_0}$. We may assume that $K\geq 1$. Since $\Vert \cdot \Vert _{\ell^2}\geq \Vert \cdot \Vert _{\ell^{\infty}}$, we have $\Vert f/\lambda \Vert _{\ell^{\infty}}\leq t_0$, hence
	\[
	\sum_{j\ge0}\phi\left(\frac{|\widehat{f}(j)|}{\lambda}\right) \leq K \frac{\Vert f \Vert _{\ell^2}}{\lambda ^2} \leq \Vert \frac{f}{\Vert f \Vert _{\ell^2}} t_0\Vert ^2 \leq 1.
	\]
	It follows that $\Vert f \Vert _{\ell^{\phi}} \leq \frac{\sqrt{K}}{t_0}\Vert f \Vert _{\ell^2}$, which gives the second inequality, with $K_{\phi}=\frac{\sqrt{K}}{t_0}$.

	Now, using that $\Vert B^k \Vert _{\ell^2}=1$ for any $k$, replacing $f$ by $B^k$ in \eqref{eqnormsequiv} and applying Proposition~\ref{prop:linfty_decay} yield $k_{\phi}\le\|B^k\|_{\ell^{\phi}}\le K_{\phi}$ for all large $k$. Adjusting the constants for the finitely many remaining values of $k$ gives Assertion~1.

	\medskip
	
	We turn to the proof of Assertion~2. Let $\phi$ be an Orlicz function such that $\phi(t)=o(t^2)$ as $t\to 0$ and let us fix $\varepsilon >0$. For $k\geq 1$, we have
	\begin{align}
		\sum _{j\geq 0}\phi\left(\frac{|\widehat{B^k}(j)|}{\varepsilon}\right) & \notag =\sum_{j\geq0}\left(\frac{|\widehat{B^k}(j)|}{\varepsilon}\right)^2\frac{\phi\left(\frac{|\widehat{B^k}(j)|}{\varepsilon}\right)}{\left(\frac{|\widehat{B^k}(j)|}{\varepsilon}\right)^2}\\
		& \leq\frac{\sup_{0< x\leq \sup_j|\widehat{B^k}(j)|}\psi(\frac{x}{\varepsilon})}{\varepsilon^2} ,	\label{eq2}
	\end{align}
	where $\psi(t)=\phi(t)/t^2$. Since $\psi(t)\to 0$ as $t\to 0$ by assumption, Proposition~\ref{prop:linfty_decay} implies that the right hand-side of $\eqref{eq2}$ goes to $0$ as $k\to \infty$, hence $\Vert B^k \Vert _{\ell_a^{\phi}} \leq \varepsilon$. This gives the second assertion of Theorem \ref{prop:lp_norms}, $\varepsilon$ being arbitrary.
	
	\medskip
	
	The proof of Assertion~3 is very similar. Assume that $\phi$ satisfies $t^2=o(\phi(t))$ as $t\to 0$ and fix $\lambda >0$. Then, with $\psi(t)=\phi(t)/t^2$ and proceeding as above, we get
	\begin{equation}\label{eq3}
		\sum _{j\geq 0}\phi\left(\frac{|\widehat{B^k}(j)|}{\lambda}\right)  \geq \frac{\inf_{0< x\leq \sup_j|\widehat{B^k}(j)|}\psi(\frac{x}{\lambda})}{\lambda^2}.
	\end{equation}
	The assumption on $\phi$ is now equivalent to $\psi(t)\to +\infty$ as $t\to 0$, so Proposition~\ref{prop:linfty_decay} implies that the right hand-side of \eqref{eq3} tends to $\infty$ as $k\to \infty$. Since $\lambda$ is arbitrary, we get the conclusion, by definition of the Luxemburg norm.
\end{proof}

\section{Proof of Lemma \ref{lemma1}}

Our alternative proof of Lemma \ref{lemma1} will be based on the following corollary of Theorem \ref{prop:lp_norms}.

\begin{lemma}\label{lemma2}Let $\phi$ be an Orlicz function such that $\phi(t)=o(t^2)$ as $t\to 0$. For any $\varepsilon>0$, any complex polynomial $P$ with $P(0)=0$ and any finite Blaschke product $B$ which is not a power of $z$, there exists $n_0\in\mathbb{N}$ such that, for any $n\geq n_0$,
	\[
	\Vert P\circ B^{n}\Vert_{\ell^{\phi}}\leq\varepsilon.
	\]
\end{lemma}

\begin{proof}Let $\varepsilon$, $P$ and $B$ be as in the statement. Let us write $P(z)=\sum_{k=1}^{d}a_{k}z^{k}$. We have $P\circ B^{n}=\sum_{k=1}^{d}a_{k}B^{nk}$, then
	\[
	\Vert P\circ B^{n}\Vert_{\ell^{\phi}}\leq\sum_{k=1}^{d}|a_{k}|\Vert B^{nk}\Vert_{\ell^{\phi}}.
	\]
	By Theorem \ref{prop:lp_norms} (2), we have $\lim_{n\to\infty}\Vert B^{nk}\Vert_{\ell^{\phi}}=0$ for all $k=1\dots d.$
	In particular
	\[
	\lim_{n\to\infty}\Vert P\circ B^{n}\Vert_{\ell^{\phi}}=0.
	\]
\end{proof}

We now turn to the proof of Lemma \ref{lemma1}.

\begin{proof}[Proof of Lemma \ref{lemma1}]Let $I$ be an arc in $\partial \D$, with $m(I)> 1- \varepsilon$. By Mergelyan's theorem, there exists a polynomial $R$ such that $\Vert R - \varphi \Vert _I \leq \varepsilon$. Let also $Q$ be a polynomial, with $\text{val}(Q)\geq v$, such that
	\[
	\Vert Q- 1\Vert _I \leq \varepsilon.
	\]
	Let $B$ be a finite Blaschke product vanishing at $0$, which is not a power of $z$, and let $n_0\in \N$ be given by Lemma \ref{lemma2}, applied to $\frac{\varepsilon}{\max(\Vert R \Vert _{\ell^1},\Vert R \Vert _I)}$, $Q$ and $B$. Setting $B_0=B^{n_0}$ and $g= R.(Q\circ B_{0})$, the triangular inequality and the invariance by translation of the $\ell^{\phi}$ norm yield $\Vert g \Vert _{\ell^{\phi}} \leq \Vert R \Vert _{\ell^1}\Vert Q\circ B_0\Vert_{\ell^{\phi}}\leq \varepsilon$. Moreover, since $B_0(0)=0$, $0$ is a zero of order at least $v$ of $g$. Let now $K=B_0^{-1}(I)$.
	Since the function $B_0$ is inner and $B_0(0)=0$, it preserves the normalized Lebesgue measure. In particular $m(K)=m(I)>1-\varepsilon$, hence $m(K\cap I)\ge 1-2\varepsilon$. Then, for any $z\in K\cap I$,
	\begin{eqnarray*}
		|R(z) Q(B_0(z)) - \varphi(z)| & \leq & |R(z)-\varphi(z)| + |R(z)||Q(B_0(z))-1| \\
		\leq 2\varepsilon.
	\end{eqnarray*}
	Observe that $m(K\cap I)\geq 1-2\varepsilon$. Since $\varepsilon$ is arbitrary, the result follows whenever we have checked that $g$ can be suitably approximated by a polynomial.
	
	Using that the dilates $g_r := z \mapsto g(rz)$ of $g$ converge to $g$ in $\ell_a^{\phi}$ and uniformly on $\overline{\D}$ as $r\to 1$, one can find $r$ close enough to $1$, such that the function $g_r$, analytic in a neighbourhood of $\overline{\D}$, satisfies the conclusion of the lemma, except it is not a polynomial. To finish, for $n\in \N$, let $S_n$ denotes the operator from $H(\D)$ to itself that maps $f=\sum_k a_kz^k$ to its $n$-th Taylor partial sum $\sum_{k=0}^na_kz^k$. Then $\Vert S_n (g_r) \Vert _{\ell^{\phi}} \leq \Vert g \Vert _{\ell^{\phi}}$ and $S_n(g_r)\to g_r$ as $n\to \infty$ uniformly on $\overline{\D}$. Thus we can set $P=S_n(g_r)$ for some $n$ large enough.
\end{proof}

\bibliographystyle{amsplain}
\bibliography{refs-CEZ-2025-09-fix}

\providecommand{\bysame}{\leavevmode\hbox to3em{\hrulefill}\thinspace}
\providecommand{\MR}{\relax\ifhmode\unskip\space\fi MR }
\providecommand{\MRhref}[2]{%
  \href{http://www.ams.org/mathscinet-getitem?mr=#1}{#2}
}
\providecommand{\href}[2]{#2}
\begin{thebibliography}{10}

\bibitem{Bayart2005}
F.~Bayart, \emph{Universal radial limits of holomorphic functions}, Glas. Math.
  J. \textbf{47} (2005), no.~2, 261--267.

\bibitem{BayartGrosseErdmanNestoridisPapadimitropoulos2008}
F.~Bayart, K.-G. Grosse-Erdmann, V.~Nestoridis, and C.~Papadimitropoulos,
  \emph{Abstract theory of universal series and applications}, Proc. London
  Math. Soc. \textbf{(3) 96} (2008), 417--463.

\bibitem{BeiseMuller2016}
H.~P. Beise and J.~M\"uller, \emph{Generic boundary behaviour of {T}aylor
  series in {H}ardy and {B}ergman spaces}, Math. Z. \textbf{284} (2016),
  1185--1197.

\bibitem{BlyudzeShimorin1996}
M.Y. Blyudze and S.M. Shimorin, \emph{Estimates of the norms of powers of
  functions in certain {B}anach spaces}, J. Math. Sci. \textbf{80} (1996),
  no.~4, 1880--1891.

\bibitem{BorichevFouchetZarouf2025I}
A.~Borichev, K.~Fouchet, and R.~Zarouf, \emph{On the {F}ourier coefficients of
  powers of a {B}laschke factor and strongly annular functions}, Constr.
  Approx. \textbf{60} (2024), no.~1, 33--86.

\bibitem{BorichevFouchetZarouf2025II}
\bysame, \emph{On the {F}ourier coefficients of powers of a finite {B}laschke
  product}, Int. Math. Res. Not. IMRN (2024), no.~20, 13255--13280.

\bibitem{Charpentier2020}
S.~Charpentier, \emph{Holomorphic functions with universal boundary behaviour},
  J. Approx. Theory \textbf{254} (2020), 105391.

\bibitem{CharpentierEspoullierZarouf2025}
S.~Charpentier, N.~Espoullier, and R.~Zarouf, \emph{Bloch functions with wild
  boundary behavior in {$\Bbb C^N$}}, Bull. Lond. Math. Soc. \textbf{57}
  (2025), no.~6, 1691--1707.

\bibitem{IvicZetaFunction2003}
A.~Ivic, \emph{The {R}iemann {Z}eta-{F}unction: {T}heory and {A}pplications},
  Dover Publications, Mineola, NY, 2003, Reprint of the 1985 original; John
  Wiley \& Sons, New York.

\bibitem{Kahane1956}
J.-P. Kahane, \emph{{S}ur certaines classes de series de {F}ourier absolument
  convergentes}, J. Math. Pures Appl. \textbf{9} (1956), 249--259.

\bibitem{KahaneNestoridis2000}
J.-P. Kahane and V.~Nestoridis, \emph{S\'eries de {T}aylor et s\'eries
  trigonom\'etriques universelles au sens de {M}enchoff}, J. Math. Pures Appl.
  \textbf{79} (2000), 855--862.

\bibitem{KahaneKatznelson1971}
J.P. Kahane and Y.~Katznelson, \emph{Sur le comportement radial des fonctions
  analytiques}, C. R. Acad. Sci. Paris S\'{e}r. A-B \textbf{272} (1971),
  A718--A719.

\bibitem{Khrushchev1978}
S.~V. Khrushchev, \emph{The problem of simultaneous approximation and of
  removal of the singularities of {C}auchy type integrals}, Trudy Mat. Inst.
  Steklov. \textbf{130} (1978), 124--195, 223.

\bibitem{Khrushchev2020}
\bysame, \emph{A continuous function with universal {F}ourier series on a given
  closed set of {L}ebesgue measure zero}, J. Approx. Theory \textbf{252}
  (2020), 105361, 6.

\bibitem{LLQRpreprint2025}
P.~Lef\`evre, D.~Li, H.~Queff\'elec, and L.~Rodr\'iguez-Piazza, \emph{On some
  questions about composition operators on weighted {H}ardy spaces}, Pure and
  Applied Functional Analysis, to appear; see arXiv:2311.01062.

\bibitem{LLQR2025}
\bysame, \emph{Characterization of weighted {H}ardy spaces on which all
  composition operators are bounded}, Analysis \& PDE \textbf{18} (2025),
  no.~8, 1921--1954.

\bibitem{Limani2024}
A.~Limani, \emph{Asymptotic polynomial approximation in the {B}loch space},
  Preprint, arXiv:2403.08723.

\bibitem{LindenstraussTzafririBookI}
J.~Lindenstrauss and L.~Tzafriri, \emph{Classical {B}anach spaces. {I}},
  Ergebnisse der Mathematik und ihrer Grenzgebiete [Results in Mathematics and
  Related Areas], Band 92, Springer-Verlag, Berlin-New York, 1977, Sequence
  spaces.

\bibitem{Maronikolakis2022}
K.~Maronikolakis, \emph{Universal radial approximation in spaces of analytic
  functions}, J. Math. Anal. Appl. \textbf{512} (2022), no.~1, 126102.

\bibitem{Meyer2021}
Y.~Meyer, \emph{Crystalline measures and mean-periodic functions}, rans. R.
  Norw. Soc. Sci. Lett. \textbf{2} (2022), 5--30.

\bibitem{SzehrZarouf2020}
O.~Szehr and R.~Zarouf, \emph{{$l_p$}-norms of {F}ourier coefficients of powers
  of a {B}laschke factor}, J. Anal. Math. \textbf{140} (2020), no.~1, 1--30.

\end{thebibliography}

\end{document}